\documentclass[12pt,a4paper]{article}

\usepackage{amsmath,amssymb}
\usepackage{amsthm}

\newtheorem{lemma}{Lemma}
\newtheorem{theorem}{Theorem}

\def\F{\mathbb F}

\def\GL{\mathop{\rm GL}}
\def\GG{\mathop{\mathcal G}}
\def\SS{\mathop{\mathcal S}}
\def\TT{\mathop{\mathcal T}}
\def\YY{\mathop{\mathcal Y}}

\def\ie{{\em i.e.}}
\def\XX{\mathcal X}

\def\be{\begin{eqnarray}}
\def\ee{\end{eqnarray}}
\def\bes{\begin{eqnarray*}}
\def\ees{\end{eqnarray*}}

\def\ll{\left(\!\!}
\def\rr{\!\!\right)}

\title{A point model for the free cyclic submodules over ternions}
\author{Hans Havlicek, Jaros\l aw Kosiorek, Boris Odehnal}

\begin{document}

\maketitle
\date{}


\begin{abstract}
\noindent We show that the set of all (unimodular and non-unimodular) free
cyclic submodules of $T^2$, where $T$ is the ring of ternions over a
commutative field, admits a point model in terms of a smooth algebraic variety.
\end{abstract}

{\em Key Words:} Ternions, projective line, cyclic submodules, Grassmann variety, Segre variety.

{\em MSC 2010:} 51B99, 51C99, 14J26, 16D40


\section{Introduction}

The present paper is devoted to the study of the geometry of free cyclic
submodules of $T^2$, where $T$ denotes the ring of ternions (upper triangular
$2\times 2$ matrices) over a commutative field $\mathbb F$. In a more geometric
language we refer to these submodules as ternionic {\em points}. We show that
the set of all such points admits a model in terms of a smooth algebraic
variety in a projective space on an $8$-dimensional vector space over $\mathbb
F$. Our exposition is based on the Grassmann variety representing the planes
($3$-dimensional subspaces) of $\mathbb F^6$ and the recent paper
\cite{havlicek+matras+pankov-11}, where a model for the set of ternionic points
in terms of planes of $\mathbb F^6$ was given and exhibited.

From \cite{havlicek+saniga-09a}, the set of ternionic points splits into two
orbits under the natural action of $\GL_2(T)$. One orbit comprises the set of
{\em unimodular points}, the elements of the other orbit are called {\em
non-unimodular points}. The variety representing the entire set of ternionic
points accordingly splits into two parts. The first arises by removing a single
line from the variety, \ie, we obtain a quasiprojective variety in the
terminology of \cite{shafarevich-94}. The second part is just that
distinguished line. So, as regards the unimodular points, our results parallel
those of A.~Herzer, who developed a very general representation theory (for
unimodular points only). See \cite[Chapt.~11 and 12]{blunck+he-05} and
\cite{herz-95} for further details. It seems worth pointing out that all points
of the distinguished line are smooth, thus giving a negative answer to the
question whether non-unimodularity of a ternionic point would imply its image
point being singular.

The geometry over ternions based on unimodular points has attracted many
authors. We refer to \cite{beck-36b}, \cite{beck-36a}, \cite{benz-77},
\cite{depu-59}, \cite{depu-60}, and the references therein. There are but a few
papers dealing with the properties of the remaining (non-unimodular) ternionic
points \cite{havlicek+matras+pankov-11}, \cite{havlicek+saniga-09a},
\cite{saniga+h+p+p-08a}, \cite{saniga+pracna-08a}.

Results and notions which are used without further reference can be found, for
example, in \cite {blunck+he-05}, \cite{herz-95}, and \cite{pankov-10a}.


\section{Main results}

Let $\F$ be a commutative field and $T$ be the ring of {\em ternions}, \ie,
the upper triangular matrices
	\be
	\ll\begin{array}{rr}
	    a_{11} & a_{12} \\
	    0      & a_{22}
	\end{array}\rr
	\label{ternion}
	\ee
with entries $a_{ij}\in\F$. Sometimes we identify $x\in\F$ with the
ternion $xI$, where $I$ is the $2\times 2$ identity matrix. $\F$ is the center
of $T$ and $T$ is a three-dimensional algebra over $\F$.

According to \cite{havlicek+saniga-09a} the non-zero cyclic submodules of the
free $T$-left module $T^2$ fall into five orbits under the natural action of
the group $\GL_2(T)$. In the following we focus on two types of submodules
given by a representative
	\be
	X_0=T\left[
	    \ll
		  \begin{array}{rr} 1 & 0\\ 0 & 1\end{array}
	    \rr,
	    \ll
		  \begin{array}{rr} 0 & 0\\ 0 & 0\end{array}
	   \rr\right]=
	\left\{
	      \left.\left[\ll
		    \begin{array}{rr} x & y\\ 0 & z\end{array}
	      \rr,
	      	\ll
		    \begin{array}{rr} 0 & 0 \\ 0 & 0\end{array}
	      \rr
	\right]\right|\ x,y,z\in\F
	\right\}
	\label{represent_X}
	\ee
and
	\be
	Y_0=T\left[
	    \ll
		  \begin{array}{rr} 0 & 0\\ 0 & 1\end{array}
	    \rr,
	    \ll
		  \begin{array}{rr} 0 & 1\\ 0 & 0\end{array}
	   \rr\right]=
	\left\{
	      \left.\left[\ll
		    \begin{array}{rr} 0 & y\\ 0 & z\end{array}
	      \rr,
	      \ll
		    \begin{array}{rr} 0 & x\\ 0 & 0\end{array}
	      \rr
	\right]\right|\ x,y,z\in\F
	\right\}.
	\label{represent_Y}
	\ee
An arbitrary submodule $X$ from the orbit of $X_0$ given in
\eqref{represent_X} is obtained by applying $S\in\GL_2(T)$ to $X_0$, \ie,
$X=X_0 S$ and will be called $X$-submodule in the following. Submodules of this
type are free and arise from unimodular pairs of $T^2$. They are the points of
the {\em projective line over $T$} as considered in \cite {blunck+he-05} and
\cite{herz-95}. This yields the general form of an $X$-submodule as
	\be
	X=\left\{
		\left.\left[
			\ll\begin{array}{cc}
	  a_{11}x & a_{12}x+a_{22}y\\ 0 & a_{22}z
	  		\end{array}\rr,
			\ll\begin{array}{cc}
	  b_{11}x & b_{12}x+b_{22}y\\ 0 & b_{22}z
	  		\end{array}\rr
		\right]\right|
	x,y,z\in\F\right\},
	\label{submodule_X}
	\ee
with the constraint
	\be
	(a_{11},b_{11})\ne(0,0)\ne(a_{22},b_{22}).
	\label{reg_cond_X}
	\ee
In the same way we obtain the general form of the $Y$-submodules using \eqref{represent_Y}
	\be
	Y=\left\{\left.\left[\ll\begin{array}{cc}
	  0 & a_{22}y+c_{22}x\\ 0 & a_{22}z\end{array}\rr,
	\ll\begin{array}{cc}
	  0 & b_{22}y+d_{22}x\\ 0 & b_{22}z\end{array}\rr\right]\right|\
	x,y,z\in\F\right\},
	\label{submodule_Y}
	\ee
with the additional condition
	\be
	a_{22}d_{22}-b_{22}c_{22}\ne0.
	\label{reg_cond_Y}
	\ee
The $Y$-submodules are precisely those free cyclic submodules which
cannot be generated by a unimodular pair. We call them {\em free non-unimodular
points}.

We identify $T^2$ with $\F^6$ by
	\bes
	\ll\begin{array}{rr|rr} a_{11} & a_{12} & b_{11} & b_{12}\\
	0 & a_{22} & 0 & b_{22}\end{array}\rr
	\mapsto
	(a_{11},b_{11},a_{22},b_{22},a_{12},b_{12})\in\F^6.
	\ees
From \eqref{submodule_X} we see that the $X$-submodules are
three-dimensional subspaces of $\F^6$ and we can extract the base vectors
	\be
	\begin{array}{ccccccccr}
	( & a_{11}, & b_{11}, &      0, &      0, & a_{12}, & b_{12}),\\
	( &      0, &      0, &      0, &      0, & a_{22}, & b_{22}),\\
	( &      0, &      0, & a_{22}, & b_{22}, &      0, &    ~~0).
	\end{array}
	\label{basepoints}
	\ee
In the following we use the Grassmann variety
$\GG\subset\F^{20}=\F^6\wedge\F^6\wedge\F^6$ representing the three-dimensional
subspaces of $\F^6$. We denote the vectors of the standard basis of $\F^{20}$
by $e_{ijk}$, $1\le i<j<k\le 6$, and write $E_{ijk}=\F e_{ijk}$ for the
respective base point of the projective coordinate system.

Every $X$-submodule is now represented by a point in the Grassmann variety $\GG$.
We compute the Grassmann
coordinates of an $X$-submodule with base points \eqref{basepoints} and the
non-vanishing coordinates read
	\be
	\begin{array}{lcrlcrlcr}
	p_{135} & = & -a_{11}a_{22}^2, &
	p_{136} & = & p_{145} = -a_{11}a_{22}b_{22}, &
	p_{146} & = & -a_{11}b_{22}^2,\\[2mm]
	p_{235} & = & -b_{11}a_{22}^2, &
	p_{236} & = & p_{245} = -b_{11}a_{22}b_{22}, &
	p_{246} & = & -b_{11}b_{22}^2,\\[2mm]
	& & & p_{356} & = &  a_{22}(a_{12}b_{22}-b_{12}a_{22}),\\[2mm]
	& & & p_{456} & = &  b_{22}(a_{12}b_{22}-b_{12}a_{22}).
	\end{array}
	\label{parametric}
	\ee
This gives a parametric representation of the Grassmann image $\XX$,
say, of the set of $X$-submodules. The parameters
$a_{11},a_{22},b_{11},b_{22}\in\F$ in \eqref{parametric} are subject to the
restrictions given in \eqref{reg_cond_X}, but there is no restriction on
$a_{12}$ and $b_{12}$. An easy calculation shows that the Grassmann image $\YY$
of the set comprising all $Y$-submodules is the line spanned by $E_{356}$ and
$E_{456}$. A parametric representation of $\YY$ is given by
	\be
	p_{356}=a_{22}(a_{22}d_{22}-b_{22}c_{22})\quad{\rm and}\quad
	p_{456}=b_{22}(a_{22}d_{22}-b_{22}c_{22}),
	\label{parametric_Y}
	\ee
where again only the non-vanishing coordinate functions are given.

We want to understand the set $\XX\cup\YY$. For that end we derive the
equations of $\XX\cup\YY$ and show the geometric meaning of both, its
parametrization and the equations as well.

At first we observe that $\XX\cup\YY$ is contained in a subspace of dimension
$8$ which is given by the linear equations
	\be
	\begin{array}{c}
	p_{123}=p_{124}=p_{125}=p_{126}=p_{134}=p_{156}=p_{234}=
			p_{256}=p_{345}=p_{346}=0,\\[2mm]
	p_{236}-p_{245}=0,\quad	p_{136}-p_{145}=0.
	\end{array}
	\label{linear}
	\ee
These are precisely the vanishing coordinate functions not mentioned in
\eqref{parametric} together with two obvious identities.

In order to describe $\XX\cup\YY$ we show:

\begin{lemma}\label{LEM_plane}
For any $(u,v)\in\F^2\setminus\{(0,0)\}$ and $i\in\{1,2\}$ let
	\be
	q_i(u,v) & = & u^2e_{i35}+uv(e_{i36}+e_{i45})+v^2e_{i46},					\label{conic_section}\\[2mm]
	r(u,v) &= & ue_{356}+ve_{456}.\label{line}
	\ee
Then $\gamma(u,v):=\F q_1(u,v)+\F q_2(u,v)+\F r(u,v)$ is a plane. As
$(u,v)$ varies in $\F^2\setminus\{(0,0)\}$ the union of these planes equals
$\XX\cup\YY$.
\end{lemma}

\begin{proof}
First we note that for $i\in\{1,2\}$ and variable $(u,v)$ the points $\F
q_i(u,v)$ comprise a conic section $c_i$ in the plane spanned by $E_{i35}$,
$\F(e_{i36}+e_{i45})$, and $E_{i46}$, whereas the points $\F r(u,v)$ form the
line spanned by $E_{356}$ and $E_{456}$. Hence any $\gamma(u,v)$ is a plane.

Given any point of $\XX$ with parameters as in \eqref{parametric} and
\eqref{reg_cond_X}, we let $u=a_{22}$ and $v=b_{22}$. Then one immediately
reads off from \eqref{parametric} that this point belongs to $\gamma(u,v)$.

Conversely, let $G$ be a point of $\gamma(u,v)$ which lies off the line $\YY$.
So we may assume $G=\F g$, where
	$$
	g=g_1q_1(u,v)+g_2q_2(u,v)+g_3r(u,v)
	$$
with $(g_1,g_2,g_3)\in\F^3$ and $(g_1,g_2)\ne(0,0)$. In order to show
that the coordinates of $G$ can be expressed as in \eqref{parametric} it
suffices to let $a_{22}=u$, $b_{22}=v$, $a_{11}=-g_1$, $b_{11}=-g_2$. Moreover,
since $(a_{22},b_{22})\ne(0,0)$ there exists at least one pair
$(a_{12},b_{12})\in\F^2$ such that $a_{12}b_{22}-b_{12}a_{22}=g_3$.

We already know from \eqref{parametric_Y} that $\YY$ is the Grassmann image of
the set of all $Y$-modules.
\end{proof}

The parametric representation of $\XX$ and $\YY$ from \eqref{parametric} and
\eqref{parametric_Y} allows to derive equations of $\XX\cup\YY$ by eliminating
parameters. From now on we restrict ourselves to the $8$-dimensional subspace
given by Eqs.\ \eqref{linear} and disregard the ambient $20$-dimensional space.
In other words, equations \eqref{linear} are always assumed to be satisfied
without further notice. We find nine quadratic equations which will be arranged
in three groups. The first two equations
	\be
	\begin{array}{c}
	p_{135}p_{146}-p_{145}^2=0,\quad p_{235}p_{246}-p_{245}^2=0
	\end{array}
	\label{cones}
	\ee
describe quadratic cones erected on the conic sections $c_i$ mentioned
in the proof of Lemma \ref{LEM_plane}. These quadratic cones have
five-dimensional subspaces for their vertices and six-dimensional generators.

Further we consider the quadratic equations
	\be
	\begin{array}{c}
	p_{146}p_{245}-p_{145}p_{246}=0,\quad
	p_{135}p_{246}-p_{235}p_{146}=0,\\[2mm]
	p_{135}p_{245}-p_{145}p_{235}=0.
	\label{segre}
	\end{array}
	\ee
They determine three quadratic cones on ruled quadrics. Together with
the linear equations $p_{356}=p_{456}=0$ the equations \eqref{segre} describe a
Segre variety $\SS$ which is the product of a line and a plane, cf.\
\cite[p.~189]{hirschfeld+thas-91}. In parametric form this Segre variety can be
written as the set of all points with coordinates
	\be
	p_{i35}=u_1v_i,\quad p_{i36}=p_{i45}=u_2v_i,\quad p_{i46}=u_3v_i
	\label{segre_parametric}
	\ee
with $i\in\{1,2\}$, $(u_1,u_2,u_3)\in\F^3\setminus\{(0,0,0)\}$, and
$(v_1,v_2)\in\F^2\setminus\{(0,0)\}$, where all other coordinates are
understood to be zero.

Finally the third set of quadratic equations reads
	\be
	\begin{array}{c}
	p_{135}p_{456}-p_{145}p_{356}=0,\quad
	p_{145}p_{456}-p_{146}p_{356}=0,\\[2mm]
	p_{235}p_{456}-p_{245}p_{356}=0,\quad
	p_{245}p_{456}-p_{246}p_{356}=0.\\
	\label{rest}
	\end{array}
	\ee
These are the equations of quadratic cones on ruled quadrics. All of them have
four-dimensional vertices, six-dimensional generators, and share the line $\YY$
spanned by the base points $E_{356}$ and $E_{456}$.

Now we can prove the following result:

\begin{theorem}
The set $\XX\cup\YY$ is a smooth algebraic variety given by the equations
\eqref{linear}, \eqref{cones}, \eqref{segre}, and \eqref{rest}.
\end{theorem}

\begin{proof}
It is easily verified that the coordinate functions of both parametric
representations, namely that of $\XX$ given in \eqref{parametric} and that of
$\YY$ given in \eqref{parametric_Y} annihilate Eqs.\ \eqref{linear},
\eqref{cones}, \eqref{segre}, and \eqref{rest}.

Conversely, we have to show that for any point $P=\F(\ldots,p_{ijk},\ldots)$
given by Eqs.\ \eqref{linear}, \eqref{cones}, \eqref{segre}, and \eqref{rest}
there are parameters such that $P$ can be written as in \eqref{parametric} or
\eqref{parametric_Y}, respectively.

We distinguish two cases: Assume first that $(p_{356},p_{456})\ne(0,0)$,
whereas all other coordinates of $P$ vanish. Then we let $a_{22}=p_{356}$ and
$b_{22}=p_{456}$. There exists at least one pair $(c_{22},d_{22})\in\F^2$ such
that \eqref{reg_cond_Y} is satisfied. Now \eqref{parametric_Y} shows that $P$
is a point on $\YY$.

Otherwise $(p_{135},p_{145},p_{146},p_{235},p_{245},p_{246})$ is a non-trivial
zero of \eqref{segre}. We infer from \eqref{segre_parametric} that there are
parameters $(u_1,u_2,u_3)\in\F^3\setminus\{(0,0,0)\}$ and
$(v_1,v_2)\in\F^2\setminus\{(0,0)\}$ such that
	\be
	(p_{i35},p_{i45},p_{i46})=(u_1v_i,u_2v_i,u_3v_i)
		\quad{\rm for}\quad i\in\{1,2\}.
	\label{segre_repara}
	\ee
Hence $(p_{i35},p_{i45},p_{i46})\ne(0,0,0)$ for at least one value of
$i$. We may assume w.l.o.g.\ that this is the case for $i=1$, whence $v_1\ne
0$. We substitute \eqref{segre_repara} in the first equation of \eqref{cones},
divide by $v_1^2$, and get the constraint $u_1u_3-u_2^2=0$ which reminds us of
the equation of a conic section. The well-known Veronese parametrization of a
conic section shows that there exists a pair
$(a_{22},b_{22})\in\F^2\setminus\{(0,0)\}$ and a constant
$k\in\F\setminus\{0\}$ such that
	\bes
	(u_1,u_2,u_3)=k(a_{22}^2,a_{22}b_{22},b_{22}^2).
	\ees
Further we define $a_{11}:=-kv_1$, $b_{11}:=-kv_2$, whence $p_{i35}$,
$p_{i45}$, and $p_{i46}$ for $i\in\{1,2\}$ are already given as in
\eqref{parametric}. Now we substitute into \eqref{rest} and obtain
	\bes
	-a_{11}a_{22}(a_{22}p_{456}-b_{22}p_{356})=0, &
		-a_{11}b_{22}(a_{22}p_{456}-b_{22}p_{356})=0,\\
	-b_{11}a_{22}(a_{22}p_{456}-b_{22}p_{356})=0, &
		-b_{11}b_{22}(a_{22}p_{456}-b_{22}p_{356})=0.
	\ees
At least one of these equations shows us that
$(p_{356},p_{456})=m(a_{22},b_{22})$ for some $m\in\F$. As
$(a_{22},b_{22})\ne(0,0)$ we can find $a_{12},b_{12}\in\F$ such that
$m=a_{12}b_{22}-b_{12}a_{22}$. Thus, finally, the coordinates of $P$ are
expressed like in \eqref{parametric} which shows $P\in\XX$.

In order to show that $\XX\cup\YY$ is smooth we compute the partial derivatives
of the parametrization given in \eqref{parametric}. Then it is a simple and
straightforward calculation that the subspace spanned by the derivatives is of
dimension $4$ at any point of $\XX\cup\YY$, regardless of the characteristic of
$\F$.
\end{proof}


\section{Final remarks}

The contents of Lemma \ref{LEM_plane} as well as the parametrization of $\XX$
given in \eqref{parametric} admit a geometric interpretation. The homogeneous
parameter $(a_{22},b_{22})\ne(0,0)$ determines a unique point $\F
q_i(a_{22},b_{22})$ on either conic section $c_i$. Further $(a_{22},b_{22})$
determines a unique point $\F r(a_{22},b_{22})$ on the line $\YY$. Any plane
$\gamma$ mentioned in Lemma \ref{LEM_plane} is spanned by these three points.
Thus there is a projective mapping from the projective line of parameters to
the planes on $\XX\cup\YY$. The homogeneous parameters
$(a_{12},b_{12})\ne(0,0)$ and $(a_{11},b_{11})\ne(0,0)$ serve as coordinates of
points of $\XX$ within the generator planes $\gamma$.

We exhibit some subrings within the ternions. Firstly, for $a_{11}=a_{22}$ and
$a_{12}=0$ in \eqref{ternion} we obtain the subring of scalar matrices. It is
clearly isomorphic to the ground field $\F$. Accordingly, by letting
$a_{11}=a_{22}$, $b_{11}=b_{22}$, and $a_{12}=b_{12}$ in \eqref{parametric}, we
obtain a parametric representation of a twisted cubic ${\mathcal F}\subset\XX$
as a model for the projective line of $\F$. Note that no $Y$-submodule can be
written in terms of scalar matrices.

Secondly, letting $a_{11}=a_{22}$ in \eqref{ternion} we find a representation
of the {\em dual numbers} over $\F$. As before, no non-unimodular points arise
from dual numbers. Hence a parametric representation of the free cyclic
submodules is obtained from \eqref{parametric} by substituting $a_{22}=a_{11}$
and $b_{22}=b_{11}$. This yields a ruled surface in a subspace of dimension
$6$. Any pair $(a_{11},b_{11})\ne(0,0)$ fixes a point on the line joining the
previously mentioned points $C_1$ and $C_2$. So we have a twisted cubic winding
about a ``tube-like surface'' $\TT\subset\SS$ whose generators are the lines
spanned by corresponding points $C_1$ and $C_2$. The cubic curve meets any
generator of this tube-like surface exactly once. The ruled surface appearing
as the point model of the free cyclic submodules over the dual numbers has two
distinguished directrices: the twisted cubic $\mathcal F$ and the line $\YY$.
There is a projective correspondence between the directrices, and corresponding
points are joined in order to form the generators of this ruled surface.

Thirdly, we put $a_{12}=0$, then formula \eqref{ternion} gives a representation
of the {\em double numbers} over $\F$. The unimodular points of the projective
line over the double numbers form the tube-like surface $\TT$ mentioned above.
This is easily seen, if we insert $a_{12}=b_{12}=0$ in \eqref{parametric}. The
surface $\TT$ is a ruled surface whose two-dimensional generators appear as the
subspaces spanned by corresponding points of $C_1$ and $C_2$. The surface $\TT$
is a subset of the Segre $\SS$, for $c_i$ is contained in the plane spanned by
$\F E_{i35}$, $\F(e_{i36}+e_{i45})$, and $\F E_{i45}$ with $i\in\{1,2\}$.
Non-unimodular points over the double numbers do not exist.


\section*{Acknowledgements}
This work was carried out within the framework of the Scientific and
Technological Cooperation Poland-Austria 2010--2011. The authors wish to thank
Andrzej Matra\'{s} (Olsztyn) for his useful remarks in the course of numerous
vivid discussions.


\noindent
Hans Havlicek and Boris Odehnal\\
Institut f\"{u}r Diskrete Mathematik und Geometrie\\
Technische Universit\"{a}t\\
Wiedner Hauptstra{\ss}e 8--10/104\\
A-1040 Wien\\
Austria\\
\texttt{havlicek@geometrie.tuwien.ac.at}\\
\texttt{boris@geometrie.tuwien.ac.at}
\par~\par
\noindent%
Jaros\l aw Kosiorek\\
Faculty of Mathematics and Computer Science\\
University of Warmia and Mazury\\
S\l oneczna 54\\
PL-10-710 Olsztyn\\
Poland\\
\texttt{kosiorek@matman.uwm.edu.pl}
\end{document}